\documentclass[12pt]{article}

\usepackage[affil-it]{authblk}
      \usepackage{latexsym}
         \usepackage[reqno, namelimits, sumlimits]{amsmath}
         \usepackage{amssymb, amsfonts}
         \usepackage{amsthm}

 \newtheorem{theorem}{Theorem}[section]
 
 \newtheorem{definition}[theorem]{Definition}
 
 \newtheorem{lemma}[theorem]{Lemma}

\title{ On blowup of nonendpoint borderline Lorentz norms for the Navier-Stokes equations
}
\author{T Barker, G Seregin
  \thanks{Email addresses: \texttt{tobias.barker@seh.ox.ac.uk, seregin@maths.ox.ac.uk}; }}
\affil{OxPDE, Mathematical Institute, University of Oxford, Oxford,UK}
\date{ \today}

\begin{document}
\maketitle
\begin{abstract}
\end{abstract}
\setcounter{equation}{0}
Assuming $T$ is a potential blow up time for the Navier-Stokes system in $\mathbb{R}^3$ or $\mathbb{R}^3_+$, we  show that the $L^{3,q}$  Lorentz norm, with $q$ finite, of the velocity field goes to infinity as time $t$ approaches $T$.\\
\setcounter{equation}{0}
\section{Introduction}
In our paper we consider the Cauchy problem for the Navier-Stokes system in the space-time domain $Q_+=\Omega\times ]0,\infty[$ for vector-valued function $v=(v_1,v_2,v_3)=(v_i)$ and scalar function $q$, satisfying the equations
\begin{equation}\label{directsystem}
\partial_tv+v\cdot\nabla v-\Delta v=-\nabla q,\qquad\mbox{div}\,v=0
\end{equation}
in $Q_+$,
the boundary conditions
\begin{equation}\label{directbc}
v=0\end{equation}
on $\partial\Omega\times [0,\infty[$,
and the initial conditions
\begin{equation}\label{directic}
v(\cdot,0)=v_0(\cdot)\in C^\infty_{0,0}(\Omega):=\{v\in C_{0}^{\infty}(\Omega): \rm{div}\, v=0\}\end{equation}
in $\Omega$. It is assumed that 
 $\Omega$ either $\mathbb{R}^3$ or $\mathbb{R}^3_{+}$.
 For $\Omega=\mathbb{R}^3$, in the classical paper \cite{Le}, Leray showed the existence of a solution $v$ to the problem (\ref{directsystem})-(\ref{directic}) satisfying the global energy inequality
 \begin{equation}\label{energyinequality}
\frac 12\int\limits_{\mathbb{R}^3} |v(x,t)|dx+\int\limits_0^t\int\limits_{\mathbb{R}^3}|\nabla v|^2dxdt'\leqslant \frac 12\int\limits_{\mathbb{R}^3}|v_0|^2 dx.
 \end{equation}
 Later on, in \cite{Hopf}, Hopf made important contributions when $\Omega$ is a bounded domain with sufficiently smooth boundary.
 To recall the modern defition of weak Leray-Hopf solutions, we first introduce certain necessary notation. 
 Let $Q_{T}:=\Omega\times ]0,T[$, $L_2(\Omega)$ and ${W}^{1}_{2}(\Omega)$ be usual Lebesgue and Lebesgue spaces, respectively. We define further 
 $\overset{\circ}{J}(\Omega)$ as the closure of $C_{0,0}^{\infty}(\Omega)$ in $L_2(\Omega)$ and 
 $\overset{\circ}{J}{^{1}_{2}}(\Omega)$ is the closure of the same set 
  with respect to the Dirichlet metric.
 
 \begin{definition}\label{weakLerayHopf}
 Let $\Omega$ be a domain in $\mathbb{R}^3$ 
  and $v_0\in\overset{\circ}{J}(\Omega)$.  A weak-Leray Hopf solution to (\ref{directsystem})-(\ref{directic}) is a vector field $v:Q_\infty\rightarrow \mathbb{R}^3$ such that 
 \begin{equation}\label{vLerayenergyspace}
 v\in L_{\infty}(0,\infty;\overset{\circ}{J}(\Omega))\cap L_{2}(0,\infty;\overset{\circ}{J}{^{1}_{2}}(\Omega));
\end{equation}
 
 the function $t\rightarrow \int\limits_{\Omega}v(x,t)\cdot w(x)dx$ is continuous at any point $t\in[0,\infty[$ for any $w\in L_2(\Omega)$;

 for any divergence free test function $w\in C_{0}^{\infty}(Q_\infty)$
 \begin{equation}\label{distributionalsolution}
 \int\limits_{Q^+}(-v\cdot\partial_t w- v\otimes v:\nabla w+\nabla v:\nabla w) dx dt=0;
 \end{equation}
 
for any $t\in[0,\infty[$  
  \begin{equation}\label{energyleray}
 \frac 12\int\limits_{\Omega} |v(x,t)|dx+\int\limits_0^t\int\limits_{\Omega}|\nabla v|^2dxdt'\leqslant \frac 12\int\limits_{\Omega}|v_0|^2 dx;
 \end{equation}
 \begin{equation}\label{strongcontzero}
 \lim_{t\rightarrow 0^+}\|v(\cdot,t)-v_0(\cdot)\|_{L_{2}(\Omega)}=0.
 \end{equation}
 \end{definition}
 The uniqueness or non-uniqueness of Leray-Hopf solutions is a long-standing open problem. It is also known for a long time that smoothness of weak Leray-Hopf solutions implies their uniqueness in the class of weak solutions. For initial data satisfying (\ref{directic}), the velocity field $v$ is smooth and bounded
 for some short interval at least, see (incomplete) references  \cite{Heywood},\cite{kiselevL}, \cite{L1967} and  \cite{Le}. The loss of regularity might happen due to a blowup of the velocity. We define a blowup time $T$ as the first moment of time when
 \begin{equation}\label{definitionblowup}
\lim\limits_{t\uparrow T}\|v(\cdot,t)\|_{L_{\infty}(\Omega)}=\infty.\end{equation}

Our aim is to prove a certain necessary conditions for $T$ be a blowup time. Let us first mention several results in this direction.

In \cite{Le} Leray proves necessary conditions for $T$ to be a blowup time for $\Omega=\mathbb{R}^3$ namely:
\begin{equation}\label{Lerayrates}
\|v(\cdot,t)\|_{s,\Omega}\geq \frac {c_s}{(T-t)^{\frac {s-3}{2s}}}
\end{equation}
for any $0<t<T$, for all $s>3$, and  for a positive constant $c_s$ depending only on $s$.
Later in \cite{Giga1986} Giga proves  (\ref{Lerayrates}) for a wide class of domains $\Omega$ including a half space and bounded domains with sufficiently smooth boundaries.
 
 For the case $s=3$  no estimate of the form (\ref{Lerayrates}) is available. 
 However, in \cite{ESS2003}, it has  been proven:
 \begin{equation}\label{limsup}
\limsup\limits_{t\uparrow T}\|v(\cdot,t)\|_{L_{3}(\Omega)}=\infty.
\end{equation}
 In papers \cite{MSh2006} and \cite{S2005}, it has been shown that
(\ref{limsup}) remains to be true for
for  $\Omega=\mathbb R^3_+:=\{x=(x_i)\in \mathbb R^3:\,\,x_3>0\}$ and for $\Omega$ being a bounded domain with sufficiently smooth boundary.

Recent progress has been made in establishing the validity of (\ref{limsup}) for other critical spaces. We refer  to $X$, consisting of measurable functions acting on domains in 
$\mathbb{R}^3$, as critical  if for $u\in X$ such that $u_{\lambda}(x)=\lambda u(\lambda x)$  we have
$\|u_{\lambda}\|_{X}=\|u\|_{X}.$
In \cite{Phuc} and \cite{WangZhang1}, it was shown that when $X$ is the non endpoint Lorentz space with $q$ finite $L^{3,q}(\mathbb{R}^3)$ condition (\ref{limsup}) remains to be true. Recently, Gallagher et al proved (\ref{limsup}) holds for $X= \dot{B}^{-1+\frac{3}{ p}}_{p,q}(\mathbb{R}^3)$ in the framework of ''strong'' solutions. Later this was shown for weak Leray-Hopf solutions in \cite{WangZhang1}.

Our work is motivated by the following question:
what are the critical spaces $X(\Omega)$ and domains $\Omega$ for which 
\begin{equation}\label{criticalblowuplimit}
\lim\limits_{t\uparrow T}\|v(\cdot,t)\|_{X(\Omega)}=\infty
\end{equation}
holds true?

In \cite{Ser12} one of the authors proves this holds for $\Omega=\mathbb{R}^3$ and $X(\Omega)=L_{3}(\Omega)$ using the theory of local energy solutions in \cite {LR1}.
Later on both authors showed in \cite{BarkerSer} this remains to be true in the case $\Omega=\mathbb{R}^3_{+}$ and $X(\Omega)=L_{3}(\Omega)$. We now claim the following:
\begin{theorem}\label{limlorentz}
Let $\Omega=\mathbb{R}$ or $\mathbb{R}^{3}_{+}$. Suppose $3\leqslant q<\infty$.
Let $T$ be a blow-up time. Then necessarily
$$\lim\limits_{t\uparrow T}\|u(\cdot,t)\|_{L^{3,q}(\Omega)}=\infty.$$
\end{theorem}
Before commenting further let us define the Lorentz spaces. 
For a measurable function $f:\Omega\rightarrow\mathbb{R}$ define:
\begin{equation}\label{defdist}
d_{f,\Omega}(\alpha):=|\{x\in \Omega : |f(x)|>\alpha\}|.
\end{equation}
Given a measurable subset $\Omega\subset\mathbb{R}^{n}$,  the Lorentz space $L^{p,q}(\Omega)$, with $p\in ]0,\infty[$, $q\in ]0,\infty]$, is the set of all measurable functions $g$ on $\Omega$ such that the quasinorm $\|g\|_{L^{p,q}(\Omega)}$ is finite. Here:

\begin{equation}\label{Lorentznorm}
\|g\|_{L^{p,q}(\Omega)}:= \Big(p\int\limits_{0}^{\infty}\alpha^{q}d_{g,\Omega}(\alpha)^{\frac{q}{p}}\frac{d\alpha}{\alpha}\Big)^{\frac{1}{q}},
\end{equation}
\begin{equation}\label{Lorentznorminfty}
\|g\|_{L^{p,\infty}(\Omega)}:= \sup_{\alpha>0}\alpha d_{g,\Omega}(\alpha)^{\frac{1}{p}}.
\end{equation}\\
It is well known that for $q\in ]0,\infty[,\,q_{1}\in ]0,\infty]$ and $q_{2}\in ]0,\infty]$ with $q_{1}\leqslant q_{2}$ we have the following embedding
$ L^{p,q_1} \hookrightarrow  L^{p,q_2}$
and the inclusion is known to be strict.
Roughly speaking, the second index of Lorentz spaces gives information regarding nature of logarithmic bumps. For example, 
for any $1>\beta>0, q>3$ we have
\begin{equation}\label{lorentzlogsing}
|x|^{-1}|\log(|x|^{-1})|^{-\beta}\chi_{|x|< 1}(x)\in L^{3,q}(\mathbb{R}^3)\,\,\,\rm{if\, and\, only\, if}\,\,q>\frac{1}{\beta}.
\end{equation}
In this way Theorem \ref{limlorentz} gives a strengthening of the previous result obtained by both authors in \cite{BarkerSer}.  It should be stressed that, at the time of writing, (\ref{criticalblowuplimit}) is open for the critical norm $L^{3,\infty}(\Omega)$ that contains $|x|^{-1}$. Furthermore, uniqueness of  weak Leray-Hopf solutions in the space 
$L_{\infty}(0,T; L^{3,\infty}(\Omega))$ remains open. We mention that  interior regularity results, that have smallness condition on $L_{\infty}(L^{3,\infty})$ norm, have been obtained in  \cite{kimkozono}, \cite{kozono} and \cite{Tsai}, for example. 

In order to describe the heuristics behind the proof of Theorem \ref{limlorentz}, we first mention previous work. In \cite{Ser12}, the method used for the statement in $L_{3}(\mathbb{R}^3)$ is, roughly speaking, as follows. It is based on assuming, for contradiction, an increasing sequence of times tending to the blow-up time such that the $L_{3}$-norm of the velocity is bounded. Then a suitable rescaling and limiting procedure is performed which gives the special type of the so-called local energy ancient solutions to the Navier-Stokes equations that coincide with Lemarie-Rieusset solutions to the Cauchy problem for same equations on some finite time interval. Those solutions have been introduced by Lemarie-Rieusset in \cite{LR1}, see also for some definitions in \cite{KS}. Moreover, the ancient solution is non-trivial. The contradiction is then obtained by proving a Liouville type theorem for those solutions based on backward uniqueness. The key point to mention from this is that in order to produce a Lemarie-Rieusset local energy solution to the Cauchy problem on a finite time interval, one needs certain behaviour of the limiting solution near it's initial time. The way to gain this in \cite{Ser12} is to split the rescaled solutions to two parts, one is the heat semigroup with weakly converging initial data and the other contains the non linearity but has zero initial data. 

However, in \cite{BarkerSer},  a different approach has been used, for both $L_{3}(\mathbb{R}^3)$ and $L_{3}(\mathbb{R}^3_{+})$. The rescaling and splitting of rescaled solutions is the same as in \cite{Ser12}, the  main difference with \cite{Ser12} is that the theory of Lemarie-Rieusset local energy solutions is not used.
Hence, heuristically speaking, the behaviour of the limit solution near it's initial time is not necessary in \cite{BarkerSer}.
Our observation is that this allows us to decompose the rescaled initial data and solutions in a different way to as was done in  \cite{BarkerSer} and \cite{Ser12}. Combining this observation and ideas from \cite{BarkerSer} allows us to strengthen the result of \cite{BarkerSer}, and thus providing Theorem \ref{limlorentz}.

\setcounter{equation}{0}
\section{Proof of Theorem 1.2}
We make use of the following notation:
$$ Q^+_{-A,0}:= \mathbb{R}^3_+\times ]-A,0[,\,\,Q_{\infty}:=\mathbb{R}^3\times ]0,\infty[,\,\,\,Q_{\infty}^+:=\mathbb{R}^3_+\times ]0,\infty[.$$
Furthermore $$B(x_0,R):=\{x: |x-x_0|<R\},\,B^+(x_0,R):=\{x\in B(x_0,R): x_3\geqslant x_{03}\}.$$
Let $T$ be a positive parameter, $\Omega$ a domain in $\mathbb{R}^3$. Then $Q_{T}:=\Omega\times ]0,T[$.
Let $L_{m,n}(Q_T)$ be the space of measurable $\mathbb{R}^l$- valued functions with the following norm
$$\|f\|_{L_{m,n}(Q_{T})}:=(\int\limits_0^T\|f(\cdot,t)\|_{L_{m}(\Omega)}^{n} dt)^{\frac 1n},$$
for $n\in[1,\infty[$, and with the usual modification if $n=\infty$.We  define the following Sobolev spaces with the mixed norm:
$$ W^{1,0}_{m,n}(Q_{T})=\{ v\in L_{m,n}(Q_{T}): \|v\|_{L_{m,n}(Q_{T})}+$$$$+\|\nabla v\|_{L_{m,n}(Q_{T})}<\infty\},$$
$$ W^{2,1}_{m,n}(Q_{T})=\{ v\in L_{m,n}(Q_{T}): \|v\|_{L_{m,n}(Q_{T})}+$$$$+\|\nabla v\|_{L_{m,n}(Q_{T})}+\|\nabla^{2} v\|_{L_{m,n}(Q_{T})}+\|\partial_{t} v\|_{L_{m,n}(Q_{T})}<\infty\}.$$

Now we state and prove  simple (but important) fact about Lorentz spaces concerning a decomposition. This will be formulated as a Lemma. Analogous statement  is Lemma II.I proven by Calderon in \cite{Calderon90}. 
\begin{lemma}\label{Decomp}
Take $1< t<r<s\leqslant\infty$, and suppose that $g\in L^{r,\infty}(\Omega)$. For any $N>0$, we let 
$g_{N_{-}}:= g\chi_{|g|\leqslant N}$ and $g_{N_{+}}:= g-g_{N_{-}}.$
Then 
\begin{equation}\label{bddpartg}
\|g_{N_{-}}\|_{L_{s}(\Omega)}^{s}\leqslant\frac{s}{s-r}N^{s-r}\|g\|_{L^{r,\infty}(\Omega)}^{r}-N^{s}d_{g}(N)
\end{equation}
if $s<\infty$ and $\|g_{N_{-}}\|_{L_{\infty}(\Omega)}\leqslant N$, and 
\begin{equation}\label{unbddpartg}
\|g_{N_{+}}\|_{L_{t}(\Omega)}^{t}\leqslant \frac{r}{r-t}N^{t-r}\|g\|_{L^{r,\infty}(\Omega)}^{r}.
\end{equation}
Moreover for $\Omega=\mathbb{R}^{3}$ or $\mathbb{R}^{3}_{+}$, if $g\in L^{r,p}(\Omega)$ with $1\leqslant p\leq\infty$ and $\rm{div}\,\,g=0$ in weak sense (also $g_3(x',0)= 0$ for half space in weak sense), then
$g= g_{1}+g_{2}$ where $g_{1}\in [C_{0,0}^{\infty}(\Omega)]^{L_{s}(\Omega)}$ with
\begin{equation}\label{g1divfree}
\|g_{1}\|_{L_{s}(\Omega)}\leqslant C(s,r,p,\|g\|_{L^{r,p}(\Omega)})
\end{equation}
and $g_{2}\in [C_{0,0}^{\infty}(\Omega)]^{L_{t}(\Omega)}$ with \begin{equation}\label{g2divfree}
\|g_{2}\|_{L_{t}(\Omega)}\leqslant C(r,t,p,\|g\|_{L^{r,p}(\Omega)}).
\end{equation}
\end{lemma}
\begin{proof}
Proof of decomposition (\ref{bddpartg})-(\ref{unbddpartg}) can be found in \cite{Mccormick}. 

Given $g$, satisfying assumptions of the lemma, we can find $g_{1_-}$ and $g_{1_+}$. We then can use the Helmoltz-Weyl decomposition $g_{1_-}=g_1+\nabla q_1$
where $g_1$ belongs to the required space with the estimate
$\|g_1\|_{L_s(\Omega)}\leq c\|g_{1_-}\|_{L_s(\Omega)}$, $\|\nabla q_1\|_{_s(\Omega)}\leq c\|g_{1_-}\|_{L_s(\Omega)}$, and
$$\int\limits_\Omega\nabla q_1\cdot \nabla \varphi dx = \int\limits_\Omega g_{1_-}\cdot \nabla \varphi dx,\quad \forall\varphi\in C^\infty_0(\mathbb R^3).$$ The same is true for the second counterpart. So, we have 
$$\int\limits_\Omega\nabla (q_1+q_2)\cdot \nabla \varphi dx=0\quad \forall\varphi\in C^\infty_0(\mathbb R^3).$$
Using properties of harmonic functions and the above global integrability of $\nabla q_1$ and $\nabla q_2$, we conclude that $\nabla (q_1+q_2)=0$. From this, from estimates (\ref{bddpartg})-(\ref{unbddpartg}), and embedding $L^{r,p}(\Omega)$ into $L^{r,\infty}(\Omega)$, we derive the required estimates  (\ref{g1divfree}) and (\ref{g2divfree}). 
\end{proof}
Now we are in a position to prove Theorem\ref{limlorentz}. Consider the more difficult case $\Omega=\mathbb R^3_+$.
Suppose that the statement of the theorem is false, then there exists an increasing sequence $t_{n}\uparrow T$ such that
\begin{equation}\label{Lorentzbdd}
M :=\sup_{n}\|u(\cdot, t_{n})\|_{L^{3,q}(\mathbb{R}^3_{+})}<\infty.
\end{equation}
We know there exists a singular point 
$x_{0}=(x'_{0},x_{03})$. Without loss of generality, we may assume that $x_0=0$. The case $x_{03}>0$ is easier. 

First examine the profile. It can be shown $u(x,t)\rightarrow u(x,T)$ for a.a $x\in\Omega$. From this and Fatou's lemma we infer
\begin{equation}\label{profiledistribution} 
d_{u(\cdot,T),\Omega}(\alpha)\leqslant \liminf_{n\rightarrow \infty} d_{u(\cdot,t_{n}),\Omega}(\alpha)' 
\end{equation}
\begin{equation}\label{profileLorentz}
\|u(\cdot,T)\|_{L^{3,q}(\mathbb{R}^3_{+})}\leqslant M.
\end{equation}
Our rescaling will be as follows:
$$ u^{(n)} (y,s):= \lambda_n u(x,t),\, p^{(n)} (y,s):= \lambda_{n}^{2} q(x,t),$$ 
where
$$ x= \lambda_{n} y,\,t=T+\lambda_{n}^{2}s,\,\lambda_{n}=\sqrt{\frac{T-t_{n}}{2}}.$$
So  smooth solutions $u^{(n)}$ and $p^{(n)}$ satisfy Navier-Stokes equations in $\mathbb{R}^{3}_{+}$ with
$$u^{(n)}_{0}(y,-2)=\lambda_{n}u(\lambda_{n}y,t_{n}).$$ 
By scale invariance of the norm in the space $L^{3,q}$, 
 we observe the following:
\begin{equation}\label{scaleinvariance}
\sup_{n}\|u^{(n)}_{0}(\cdot,-2)\|_{L^{3,q}(\mathbb{R}^{3}_{+})}=M<\infty.
\end{equation}
The key observation for Lorentz spaces as as follows.
From  Lemma \ref{Decomp} (along with fact embedding into weak $L_{3}$) we may decompose:
$$u_{0}^{(n)}(\cdot,-2):= u_{0}^{1,(n)}(\cdot,-2)+u_{0}^{2,(n)}(\cdot,-2),$$
where 
\begin{equation}\label{decobspace}
u_{0}^{1,n}(\cdot,-2)\in [C^{\infty}_{0,0}(\mathbb{R}^{3}_{+})]^{L_{\frac{10}{3}}(\mathbb{R}^{3}_{+})},\,u_{0}^{2,n}(\cdot,-2)\in [C^{\infty}_{0,0}(\mathbb{R}^{3}_{+})]^{L_{2}(\mathbb{R}^{3}_{+})}
\end{equation}
\begin{equation}\label{decompest}
\|u_{0}^{1,(n)}(\cdot,-2)\|_{L_{\frac{10}{3}}(\mathbb{R}^{3}_{+})}+\|u_{0}^{2,(n)}(\cdot,-2)\|_{L_{2}(\mathbb{R}^{3}_{+})}\leqslant c(M,q).
\end{equation}
We then decompose $u^{(n)}=v^{1,(n)}+v^{2,(n)}$ 
where $v^{1,(n)}$ and $p^{1,(n)}$ solve the linear problem
$$\partial_tv^{1,n}-\Delta v^{1,(n)}=-\nabla p^{1,(n)},\qquad \mbox{div}\,v^{1,n}=0$$
in $Q^+_{-2,0}$,
$$v^{1,(n)}(x',0,t)=0$$
for $(x',t)\in\mathbb R^2\times [-2,0]$,
$$v^{1,(n)}(\cdot,-2)=u_{0}^{1,(n)}(\cdot,-2)\in L_\frac {10}3(\mathbb R^3_+).$$
Using Solonnikov estimates for the Green function in a half-space one sees  that the following estimates are valid for $v^{1,n}$ ($\frac{10}{3}\leqslant l\leqslant \infty, k=1,2,\ldots$):
\begin{equation}\label{semigroupest}
\|\nabla^{k}v^{1,(n)}(\cdot,t)\|_{L_{l}(\mathbb{R}^3_{+})}\leqslant \frac{c(M)}{(t+2)^{\frac{k}{2}+\frac{3}{2}(\frac{3}{10}-\frac{1}{l})}}.
\end{equation}
Thus $$\|v^{1,(n)}(\cdot,t)\|_{L_{5}(\mathbb{R}^3_{+})}\leqslant \frac{c(M)}{(t+2)^{\frac{3}{20}}},\quad
\|v^{1,(n)}(\cdot,t)\|_{L_{4}(\mathbb{R}^3_{+})}\leqslant \frac{c(M)}{(t+2)^{\frac{3}{40}}}.$$
It is then seen that:
\begin{equation}\label{L4,5}
 \|v^{1,(n)}\|_{L_{5}(Q_{-2,0}^+)}+\|v^{1,(n)}\|_{L_{4}(Q_{-2,0}^+)}+\sup\limits_{t\in ]-2,0[}\|v^{1,(n)}(\cdot,t)\|_{L_{\frac {10}{ 3}}(R^3_+)}\leqslant c(M).
\end{equation}

The second counterpart of $u^{(n)}$ satisfies the  non-linear system
$$\partial_tv^{2,(n)}+\mbox{div}\,u^{(n)}\otimes u^{(n)}-\Delta v^{2,(n)}=-\nabla p^{2,(n)},\qquad \mbox{div}\,v^{2,(n)}=0$$
in $Q^+_{-2,0}$, the boundary conditions 
$$v^{2,(n)}(x',0,t)=0$$
for $(x',t)\in\mathbb R^2\times [-2,0]$, and the initial conditions
$$v^{2,(n)}(\cdot,-2)= u^{2,n}_{0}(\cdot,-2)$$
in $\mathbb{R}^3_+$.

The standard energy approach to the second system gives
$$ \partial_t\|v^{2,(n)}(\cdot,t)\|^2_{L_{2}(\mathbb R^3_+)}+2\|\nabla v^{2,(n)}(\cdot,t)\|^2_{L_{2}(\mathbb R^3_+)}=$$
$$= 2\int\limits_{\mathbb R^3_+}u^{(n)}\otimes u^{(n)}:\nabla v^{2,(n)}dx ds=I_1+I_2+I_3+I_4,$$
where
$$I_1=2\int\limits_{\mathbb R^3_+}v^{1,(n)}\otimes v^{1,(n)}:\nabla v^{2,(n)}dx,\quad I_2=2\int\limits_{\mathbb R^3_+}v^{1,(n)}\otimes v^{2,(n)}:\nabla v^{2,(n)}dx
$$ and $I_3=I_4=0$.

Next, let us  consequently evaluate terms on the right hand side of the energy identity. For the first term, we have
$$|I_1|\leq c\|v^{1,(n)}(\cdot,t)\|^2_{L_{4}(\mathbb{R}^3_+)}\|\nabla v^{2,(n)}(\cdot,t)\|_{L_{2}(\mathbb R^3_+)}.$$
The second term can be treated as follows:
$$|I_2|\leq c\|v^{1,(n)}(\cdot,t)\otimes v^{2,(n)}(\cdot,t)\|_{L_{2}(\mathbb R^3_+)}\|\nabla v^{2,(n)}(\cdot,t)\|_{L_{2}(\mathbb R^3_+)}\leq$$
$$\leq c\|v^{2,(n)}(\cdot,t)\|_{L_{5}(\mathbb R^3_+)}\|v^{2,(n)}(\cdot,t)\|_{L_{\frac {10}{3}}(\mathbb R^3_+)}\|\nabla v^{2,(n)}(\cdot,t)\|_{L_{2}(\mathbb R^3_+)}.
$$
Applying the known multiplicative inequality to the second factor in the right hand side of the latter bound, we find
$$|I_2|\leq c\|v^{1,(n)}(\cdot,t)\|_{L_{5}(\mathbb R^3_+)}\|v^{2,(n)}(\cdot,t)\|^{\frac 25}_{L_{2}(\mathbb R^3_+)}\|\nabla v^{2,(n)}(\cdot,t)\|^\frac 85_{L_{2}(\mathbb R^3_+)}.$$
Letting
$$y(t):=\|v^{2,(n)}(\cdot,t)\|^2_{L_{2}(\mathbb R^3_+)}$$
and using the Young inequality, we find
$$y'(t)+\|\nabla v^{2,(n)}(\cdot,t)\|^2_{L_{2}(\mathbb R^3_+)}\leq
c\|v^{1,(n)}(\cdot,t)\|^5_{L_{5}(\mathbb R^3_+)}y(t)+c\|v^{1,(n)}(\cdot,t)\|^4_{L_{4}(\mathbb R^3_+)}.$$
Next, elementary arguments lead to the inequality
$$\Big(y(t)\exp{\Big(-\int\limits^t_{-2}}\|v^{1,(n)}(\cdot,s)\|^5_{L_{5}(\mathbb R^3_+)}ds\Big)\Big)'\leq$$$$\leq c\exp{\Big(-\int\limits^t_{-2}}\|v^{1,(n)}(\cdot,s)\|^5_{L_{5}(\mathbb R^3_+)}ds\Big)\|v^{2,(n)}(\cdot,t)\|^4_{L_{4}(\mathbb R^3_+)}.$$
So,
$$y(t)\leq c\int\limits^t_{-2}\exp{\Big(\int\limits^t_\tau}\|v^{1,(n)}(\cdot,s)\|^5_{L_{5}(\mathbb R^3_+)}ds\Big)\|v^{1,(n)}(\cdot,\tau)\|^4_{L_{4}(\mathbb R^3_+)}d\tau+$$$$+
\|u_{0}^{2,(n)}(\cdot,-2)\|_{L_{2}(\mathbb{R}^{3}_{+})}^{2}$$
Using, (\ref{decompest}) and (\ref{L4,5}) it is easily seen that
\begin{equation}\label{bddenergy}
\sup_{-2<t<0}y(t)\leqslant C(M),\,\,\quad\|\nabla v^{2,(n)}\|_{L_{2}(Q_{-2,0}^+)}\leqslant C(M).
\end{equation}

From these estimates and from the multiplicative inequality,
one can deduce that
\begin{equation}\label{globalu2}
\|v^{2,(n)}\|_{L_{s}(Q^+_{-2,0})}\leq C(s,M)\end{equation}with any $s\in [2,\frac {10}3]$.
Moreover,
\begin{equation}\label{globalu2}
\|u^{(n)}\|_{L_{\frac{10}{3}}(Q^+_{-2,0})}\leq C(M).
\end{equation}
We now sketch the plan as to how to obtain the contradiction. Full details are found in \cite{BarkerSer}.
Let $$f^{(n)}= u^{(n)}.\nabla u^{(n)}+ v^{1,(n)}+v^{2,(n)}.$$
Using (\ref{semigroupest}) and (\ref{bddenergy}), one can decompose  $$f^{(n)}=f^{(n)}_1+f^{(n)}_2+f^{(n)}_3$$
with estimates 
\begin{equation}\|f^{(n)}_1\|_{L_{\frac 9 8,\frac 3 2}(Q_{-\frac 7 4,0}^+)}+\|f^{(n)}_2\|_{L_{\frac {10} {3}}(Q_{-\frac 7 4,0}^+)}+\|f^{(n)}_3\|_{L_{2}(Q_{-\frac 7 4,0}^+)}\leqslant C(M).
\end{equation}
Multiplying $u^{(n)}$ and $p^{(n)}$ by a suitable cut off function in time and using of Solonnikov coercive estimates for the linear Stokes system  (see \cite{Sol1973}-\cite{Sol2003UMN}), we can decompose $$u^{(n)}= \sum_{i=1}^3 u^{(n)}_{i},\,p^{n}= \sum_{i=1}^3 p^{(n)}_{i}$$ on $Q_{-\frac 3 2,0}^+$ with the following estimates
\begin{equation}\label{coerciveest}
\|u^{(n)}_1\|_{W^{2,1}_{\frac 9 8,\frac 3 2}(Q_{-\frac 3 2,0}^+)}+\|u^{(n)}_2\|_{W^{2,1}_{\frac {10} {3}}(Q_{-\frac 3 2,0}^+)}+\|u^{(n)}_3\|_{W^{2,1}_{2}(Q_{-\frac 3 2,0}^+)}+$$$$+
\|\nabla p^{(n)}_1\|_{L_{\frac 9 8,\frac 3 2}(Q_{-\frac 3 2,0}^+)}+\|\nabla p^{(n)}_2\|_{L_{\frac {10} {3}}(Q_{-\frac 3 2,0}^+)}+\|\nabla p^{(n)}_3\|_{L_{2}(Q_{-\frac 3 2,0}^+)}
\leqslant C(M). 
\end{equation}
One then shows $(u^{n},p^{n})$ tends to a suitable weak solution $(u,p)$ of the Navier-Stokes system on $Q_{-\frac 3 2,0}^+$ with estimate (\ref{globalu2}) and with analogous decomposition and estimates to (\ref{coerciveest}). One uses suitable epsilon regularity criterion developed in \cite{CKN}, \cite{S3} and \cite{Ser09} to show that $u$ is in fact non trivial in $Q^+(a)$ for all $a$ sufficiently small  as a by product of the assumption that the origin is a singular point at time $T$. Moreover, the previously mentioned global estimates for suitable decompositions of $(u,p)$ imply 
\begin{equation}\label{ubddexterior}
 |u(x,t)|+|\nabla u(x,t)|\leq c_1(\delta),
 \end{equation}
for all $(x,t)\in (\mathbb{R}^3_{+\delta}\setminus B^{+}(R_{1}) \times ]-\frac{5}{4},0[$. Here, $\mathbb{R}^3_{+\delta}:=\mathbb{R}^3_{+}\cap \{x_3>\delta\}$.\\
Let us briefly mention minor difference concerning zero endpoint of the limit solution in the context of Lorentz spaces.
It can be shown (as in \cite{BarkerSer}) that
$$
\int\limits_{B^{+}(a)}|u^{(n)}(x,0)| dx\rightarrow \int\limits_{B^{+}(a)}|u(x,0)| dx.
$$
Using generalized Holder inequality for Lorentz spaces (along with scale invariance), it is not so difficult to show
$$\frac{1}{a^{2}}\int\limits_{B^{+}(a)}|u^{(n)}(x,0)| dx\leqslant c\|u^{(n)}(\cdot,0)\|_{L^{3,q}(B^{+}(a))}=c\|u(\cdot,T)\|_{L^{3,q}(B^{+}(\lambda_{n}a))}.$$
Now $u(\cdot,T)\in L_{3,q}(\mathbb{R}^{3}_{+})$ and obviously $$d_{u(\cdot,T),B^{+}(\lambda_{n}a))}(\alpha)\rightarrow 0.$$
So for $0<q<\infty$, $\|u(\cdot,T)\|_{L^{3,q}(B^{+}(\lambda_{n}a))}\rightarrow 0$.\\
With this in mind one, then considers the vorticity equation 
$$\partial_{t}\omega-\Delta\omega= \rm{\,div}\,(\omega\otimes u-u\otimes\omega).$$ 
Then contradiction is then reached by showing that the limit function is trivial as a consequence of showing that $\omega=0$.
 This is shown by properties of the limit functions $(u,p)$ together with backward uniqueness and unique continuation through spatial boundaries for the heat operator with lower order terms as in \cite{ESS2003}. 
$\Box$

\end{document}